\RequirePackage{amsmath}
\documentclass[dvipsnames,runningheads]{llncs}
\usepackage[T1]{fontenc}
\usepackage{lmodern}
\usepackage{graphicx}
\usepackage{commands}
\usepackage[hidelinks]{hyperref}
\usepackage{color}
\usepackage[export]{adjustbox}

\begin{document}
\title{Orientation Scores should be a Piece of Cake}
%
\author{
Finn M. Sherry\inst{1}
\and
Chase van de Geijn\inst{2}
\and
Erik J. Bekkers\inst{3}
\and 
Remco Duits\inst{1}
}
\authorrunning{F.M. Sherry et al.}
%
\institute{
CASA \& EAISI, Dept. of Mathematics \& Computer Science, Eindhoven University of Technology, the Netherlands\\
\email{\{f.m.sherry,r.duits\}@tue.nl}
\and
Institute of Computer Science, Georg-August University Göttingen, Germany \\
\email{\{chase.geijn\}@uni-goettingen.de}
\and
AMLab, Informatics Insitute, University of Amsterdam, the Netherlands\\
\email{\{e.j.bekkers\}@uva.nl}
}
\maketitle

\begin{abstract}
We axiomatically derive a family of wavelets for an orientation score, lifting from position space $\Rtwo$ to position and orientation space $\Rtwo \times S^1$, with fast reconstruction property, that minimise position-orientation uncertainty.
We subsequently show that these minimum uncertainty states are well-approximated by cake wavelets: for standard parameters, the uncertainty gap of cake wavelets is less than $1.1$, and in the limit, we prove the uncertainty gap tends to the minimum of $1$.
Next, we complete a previous theoretical argument that one does not have to train the lifting layer in (PDE-)G-CNNs, but can instead use cake wavelets.
Finally, we show experimentally that in this way we can reduce the network complexity and improve the neurogeometric interpretability of (PDE-)G-CNNs, with only a slight impact on the model's performance.

\end{abstract}

\section{Introduction}\label{sec:introduction}
In recent years, more interpretable and robust alternatives to convolutional neural networks (CNNs) have been developed, including Group Equivariant CNNs (G-CNNs) by Cohen et al. \cite{Cohen2016GroupNetworks}, which make use of regular representations to perform equivariant convolutions. These were generalised by Smets et al. \cite{Smets2022PDENetworks} to PDE-based G-CNNs (PDE-G-CNNs), where the kernels solve parametrised PDEs well-known from classical image analysis. Of particular interest are networks that are equivariant w.r.t. the roto-translation group $\SE(2)$: the kernels (cf. Fig.~\ref{fig:ball}) are then additionally related to the concept of association fields from neurogeometry \cite{Citti2006CorticalSpace,Petitot2003NeurogeometryStructure}, as shown in \cite{Bellaard2023AnalysisPDEGCNNs}. In this way, PDE-G-CNNs further improve interpretability. However, the input data are typically images, which live on the plane instead of the group: the data must therefore first be lifted. In existing (PDE-)G-CNNs, this is done with a trainable G-CNN lifting layer. We aim to enhance the interpretability of these models by fixing the lifting layer.

For this, we take inspiration from classical multi-orientation image processing, in particular the orientation score transform \cite{Duits2007ImageGroup}. The orientation score transform uses an anisotropic wavelet to lift images on $\Rtwo$ to orientation scores on $\SE(2)$ (cf. Fig.~\ref{fig:disentanglement}). We can then interpret $\SE(2)$ as the space of positions and orientations on the plane. Requiring so-called fast approximate reconstruction puts constraints on the wavelet, and Duits et al. proposed so-called \emph{cake wavelets} (cf. Fig.~\ref{fig:cakewavelet}) \cite{Duits2007ImageGroup}. 
In this work, we show that by imposing additional constraints one derives a family of wavelets that are efficiently approximated by cake wavelets. Specifically, we impose an $\SE(2)$ minimum uncertainty constraint, in the style of the $\SIM(2)$ constraint considered by Antoine et al. \cite{Antoine1998DirectionalPatterns}. Combined with polar separability in the Fourier domain, we can extend the minimum uncertainty principle from the circle, as investigated by Barbieri et al. \cite{Barbieri2012UncertaintyArchitecture}, to the plane.

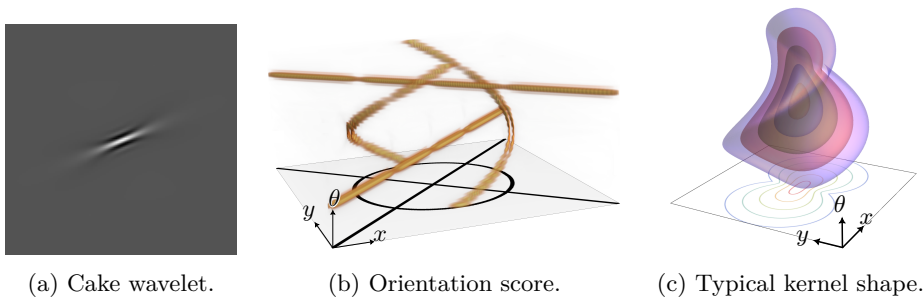
\begin{figure}
\centering
\begin{subfigure}[t]{0.25\textwidth}
\centering
\includegraphics[width=\textwidth]{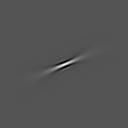}
\caption{Cake wavelet.}\label{fig:cakewavelet}
\end{subfigure}
~ 
\begin{subfigure}[t]{0.4\textwidth}
\begin{tikzpicture}
\draw(0, 0)node[inner sep=0]{\includegraphics[width=\textwidth]{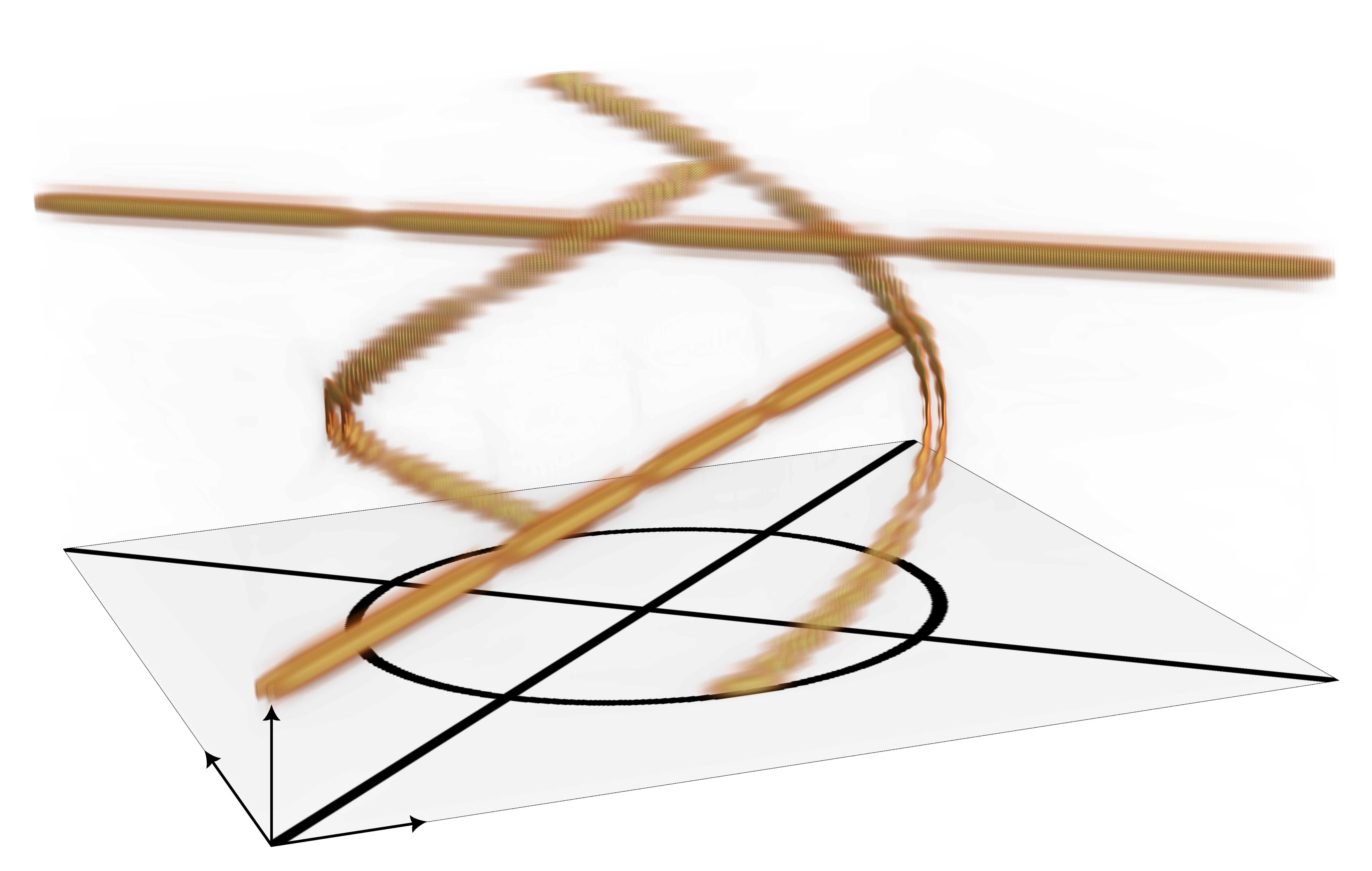}};
\draw(-0.85, -1.3)node{\contour{white}{$x$}};
\draw(-1.79, -1.0)node{\contour{white}{$y$}};
\draw(-1.45, -0.8)node{\contour{white}{$\theta$}};
\end{tikzpicture}
\caption{Orientation score.}\label{fig:disentanglement}
\end{subfigure}
~ 
\begin{subfigure}[t]{0.3\textwidth}
\begin{tikzpicture}
\draw(0, 0)node[inner sep=0]{\includegraphics[width=\textwidth]{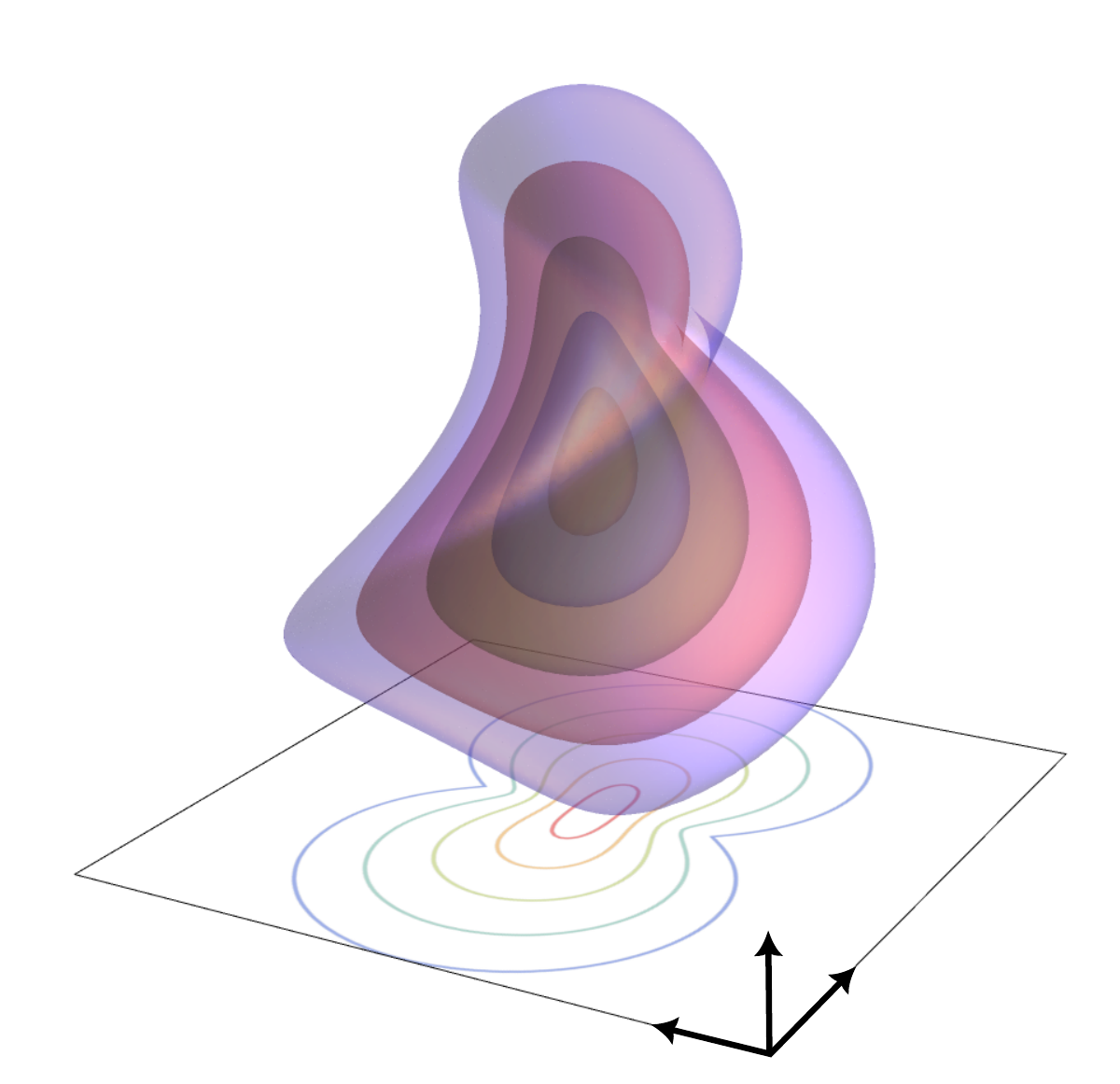}};
\draw(1.0, -1.3)node{\contour{white}{$x$}};
\draw(0.16, -1.55)node{\contour{white}{$y$}};
\draw(0.65, -1.1)node{\contour{white}{$\theta$}};
\end{tikzpicture}
\caption{Typical kernel shape.}\label{fig:ball}
\end{subfigure}
\caption{(a) Cake wavelet for orientation $\theta = \frac{\pi}{8}$. (b) Lifting disentangles crossing and overlapping structures. (c) Typical shape of a kernel in a PDE-G-CNN.}
\label{fig:summary}
\end{figure}

\subsubsection{Contribution.} We propose four axioms as desiderata for a fixed SE(2) lifting kernel: (1) polar separability, (2) fast reconstruction, (3) directionality, (4) minimum uncertainty. From these axioms, we derive a family of wavelets for orientation score transforms show that these are well-approximated by cake wavelets \cite{Duits2007ImageGroup}: for standard parameters, the uncertainty gap of cake wavelets is less than $1.1$ (Fig.~\ref{fig:convergence}). Moreover, we prove that in the limit the uncertainty gap tends to the minimum of $1$ (Thm.~\ref{thm:convergence}). We then show that (PDE-)G-CNNs that lift with cake wavelets can perform on par with networks using a trained lifting layer, while needing fewer parameters (Thm.~\ref{thm:no_need_to_train_the_lifting_layer} and Fig.~\ref{fig:experiments}). We furthermore argue that interpretability of (PDE)-G-CNNs is improved: instead of training kernels on $\Rtwo$, we train convections on $\SE(2)$, which centre \cite[Prop.~3]{Bellaard2023AnalysisPDEGCNNs} association fields from neurogeometry \cite{Citti2006CorticalSpace,Petitot2003NeurogeometryStructure} trained in subsequent PDE-G-CNN \cite{Smets2022PDENetworks} layers.
\section{Orientation Score Transform}\label{sec:orientation_score_transform}
\begin{definition}[Special Euclidean Group]\label{def:SE2}
The 2D \emph{special Euclidean group} is defined as the Lie group $\SE(2) \coloneqq \Rtwo \rtimes \SO(2)$ of roto-translations on two dimensional Euclidean space, with group product
$(\vec{x}, R) (\vec{y}, S) = (\vec{x} + R \vec{y}, R S)$.
$\SE(2)$ acts on $\Rtwo$ by $U_{(\vec{x}, R)} \vec{y} \coloneqq \vec{x} + R \vec{y}$. 
Since $\SO(2) \cong S^1$, we can identify rotations and angles; we write $R_\theta$ for the counter-clockwise rotation by angle $\theta$.
\end{definition}
This action induces the \emph{quasi-regular representation} on functions on $\Rtwo$:
\begin{definition}[Quasi-Regular Representation]\label{def:representation_U}
The \emph{quasi-regular representation} $\mathcal{U}: \SE(2) \to B(\Ltwo(\Rtwo))$ is the Lie group homomorphism defined for all $g \in \SE(2)$ and $f \in \Ltwo(\Rtwo)$ as 
$\mathcal{U}_g f \coloneqq f \after U_g^{-1}$. 
The representation $\hat{\mathcal{U}} = \mathcal{F} \after \,\mathcal{U} \after \mathcal{F}^{-1}$\!, i.e. its Fourier transform, is then given by
$\hat{\mathcal{U}}_{(\vec{x}, R)} \hat{f}(\vec{\omega}) = \exp(-i \vec{\omega} \cdot \vec{x}) \hat{f}(R^{-1} \vec{\omega})$.
\end{definition}
We want to use the orientation score transform to lift an image defined on the space of positions $\Rtwo$ to the space of positions and orientations $\SE(2)$. In practice 
we use the subgroup $\SE(2, N) \coloneqq \Rtwo \rtimes \SO(2, N)$, where $\SO(2, N)$ are the rotations that fix a regular $N$-gon. 
\begin{definition}[Orientation Score Transform]\label{def:orientation_score_transform}
The \emph{orientation score transform} $\mathcal{W}_\psi: \Ltwo(\Rtwo) \to \Ltwo(\SE(2, N))$, with \emph{proper wavelet} $\psi$, is defined by \cite{Duits2007ImageGroup}
\begin{equation}\label{eq:orientation_score_transform}
\mathcal{W}_\psi f (\vec{x}, \theta) \coloneqq (\mathcal{U}_{(\vec{x}, \theta)} \psi, f) = \int_{\Rtwo} \overline{\psi(R_{\theta}^{-1} (\vec{y} - \vec{x}))} f(\vec{y}) \diff \vec{y}
\end{equation}
for $(\vec{x}, \theta) \in \SE(2, N)$. We then call $\mathcal{W}_\psi f$ the \emph{orientation score} of image $f$.
\end{definition}
See Fig.~\ref{fig:disentanglement} for a typical example of an orientation score. 
The orientation score transform is well-defined if the wavelet is \emph{proper}, i.e. $\psi \in \Lone(\Rtwo) \intersection \Ltwo(\Rtwo)$ and the corresponding orientation score transform has a stable inverse. For stable inversion we limit the domain to $\Ltwo^{\rho_0}(\Rtwo) \coloneqq \{f \in \Ltwo(\Rtwo) \mid \supp \hat{f} \subset B(\vec{0}, \rho_0)\}$. The inversion is then stable ($\cond(\mathcal{W}_\psi) \leq \frac{M}{\delta}$) on $\mathcal{W}_\psi(\Ltwo^{\rho_0}(\Rtwo))$ if $0 < \delta \leq M_\psi(\omega) \leq M < \infty$ on $B(\vec{0}, \rho_0)$, where 
$M_\psi(\vec{\omega}) \coloneqq \sum_{\theta \in \SO(2, N)} \abs{\hat{\psi}(R_\theta^{-1} \vec{\omega})}^2$ \cite{Duits2007ImageGroup}.
\begin{definition}[Cake Wavelets]\label{def:cake_wavelets}
For $k \in \N$, we define the \emph{$k$-th order B-spline} $B_k$ as the density function of the sum of $k + 1$
i.i.d.
$\Uniform[-\frac{1}{2}, \frac{1}{2}]$ random variables modulo $2\pi$. Hence, 
$B_0 = \vec{1}\left[-\frac{1}{2}, \frac{1}{2}\right]$, and $B_k = B_0 *^{(k)} B_0$, with $*$ convolution on the circle.
Then, we define the \emph{cake wavelets} \cite{Duits2007ImageGroup} for $N \in \N$ orientations with $k$-th order B-splines as
\begin{equation} \label{eq:cake_wavelets}
\hat{\psi}(\rho \cos(\phi), \rho \sin(\phi)) \coloneqq \mathcal{M}(\rho) \Phi_{N, k}(\phi) \coloneqq \vec{1}_{B(\vec{0}, \rho_0)}(\rho) B_k\left(\frac{N}{2 \pi} \phi\right).
\end{equation}
\end{definition}
See Fig.~\ref{fig:cakewavelet} for a typical cake wavelet. We will now axiomatically derive a family of lifting wavelets, and show that they are well-approximated by cake wavelets. 
\begin{desiderata}[Wavelet Constraints]\label{des:wavelet_constraints}
We have the following desiderata for the orientation score transform, and corresponding constraints on the wavelet:
\begin{enumerate}
\item  \textbf{Polar separability in the Fourier domain}, i.e. 
\begin{equation} \label{eq:polar_separability}
\hat{\psi}(\rho \cos(\phi), \rho \sin(\phi)) = \mathcal{M}(\rho) \Phi(\phi).
\end{equation}
\item \textbf{Fast reconstruction:} $f \approx \sum_{\theta \in \SO(2, N)} \mathcal{W}_\psi f (\cdot, \theta)$ for all $f \in \Ltwo^{\rho_0}(\Rtwo)$
$\implies$ $N_\psi(\vec{\omega}) \approx 1$ on $B(\vec{0}, \rho_0)$, where 
$N_\psi(\vec{\omega}) \coloneqq \sum_{\theta \in \SO(2, N)} \hat{\psi}(R_\theta^{-1} \vec{\omega})$.
\item \textbf{Directionality:} $\hat{\psi}$ is centred around
the $y$-axis, i.e.
$\int_\R \omega^1 \abs{\hat{\psi}(\omega^1, \omega^2)}^2 \diff \omega^1 = 0$.
\end{enumerate}
\end{desiderata}
Polar separability is necessary for a well-posed uncertainty principle as we will see later. 
While the orientation score transform can be inverted for any proper wavelet, the exact inversion formula is expensive to evaluate \cite{Duits2007ImageGroup}. After processing in the space of positions and orientations, we invariably have to map the orientation score back to an image, so that the fast reconstruction property is highly desirable. 
Our notion of directionality is related to, but distinct from, that in \cite{Antoine1998DirectionalPatterns}. By centering the wavelet around the $y$-axis in the Fourier domain, it will be oriented along the $x$-axis in the spatial domain. Consequently, it will associate the \enquote{correct} orientation with an oriented feature, e.g. $\theta = \pi/2$ for vertical features. 
\begin{corollary}\label{cor:polar_profiles}
We must choose $\mathcal{M}(\rho) \approx 1$ on $[0, \rho_0]$ and $\sum_{\theta \in \SO(2, N)} \Phi(\phi - \theta) \approx 1$ on $S^1$ to satisfy the desiderata.
\end{corollary}
\begin{proof}
The first two desiderata give, on $B(\vec{0}, \rho_0)$,
\begin{equation*}
1 \approx \sum_{\theta \in \SO(2, N)} \hat{\psi}(\rho, \phi - \theta) = \mathcal{M}(\rho) \sum_{\theta \in \SO(2, N)} \Phi(\phi - \theta) = C \mathcal{M}(\rho),
\end{equation*}
for some constant $C$. We can w.l.o.g. that assume $C = 1$ by rescaling $\Phi$. \qed
\end{proof}
Cor.~\ref{cor:polar_profiles} fixes the radial profile of the wavelet in the Fourier domain to that of cake wavelets \eqref{eq:cake_wavelets}. Additionally, cake wavelets 
satisfy the constraint on the angular profile imposed by fast reconstruction. Next we add an extra desideratum -- minimising position-orientation uncertainty -- and restrict to a small 
set of wavelets.
\section{Minimum Uncertainty}\label{sec:minimum_uncertainty}
Gabor initiated the study of uncertainty in wavelet analysis in 1946 \cite{Gabor1946TheoryCommunication}: inspired by quantum mechanics, he showed that no wavelet can simultaneously have a certain time and frequency, and that a family of wavelets -- now called Gabor wavelets -- minimise the product of uncertainties. We can view these wavelets as coherent states of the generators of the Heisenberg group. 
Other authors have derived similar results. 
For instance, Antoine et al. \cite{Antoine1998DirectionalPatterns} showed that Cauchy wavelets minimise uncertainty with respect to the translation-rotation-scaling group $\SIM(2) \coloneqq \Rtwo \rtimes (\SO(2) \times \R_{>0})$.
Christensen et al. showed that the uncertainty principle on the circle by Breitenberger \cite{Breitenberger1983UncertaintyObservables} can be derived from an uncertainty principle on $\SE(2)$ \cite{Christensen2004UncertaintyGroup}.
Barbieri et al. obtained the corresponding minimum uncertainty states \cite{Barbieri2012UncertaintyArchitecture} using the irreducible representations of $\SE(2)$ thereon while investigating a model of the visual cortex. 
We first introduce the required machinery. 
\begin{definition}\label{def:expectation_and_variance}
Let $X$ be an operator on Hilbert space $H$, and $\psi \in \mathcal{D}(X)$.
We define \emph{expectation} $\expectation{X}{\psi} \coloneqq (\psi, X \psi)$ and \emph{variance} $\expectation{(X - \expectation{X}{\psi})^2}{\psi}$.
\end{definition}
\begin{lemma}[Uncertainty Principle (UP)]\label{lem:uncertainty_principle}
Let $A$, $B$ be symmetric operators on Hilbert space $H$. Let $[A, B] = AB - BA$ be the \emph{commutator}, with $\mathcal{D}([A, B]) \coloneqq \{\psi \in \mathcal{D}(A) \intersection \mathcal{D}(B) \mid A \psi \in \mathcal{D}(B), B \psi \in \mathcal{D}(A)\}$. 
The \emph{uncertainty gap} is then \cite{Griffiths}
\begin{equation}\label{eq:uncertainty_principle}
\UG_\psi(A, B) \coloneqq \frac{\expectation{(A - \expectation{A}{\psi})^2}{\psi} \expectation{(B - \expectation{B}{\psi})^2}{\psi}}{\frac{1}{4} \abs*{\expectation{[A, B]}{\psi}}^2} \geq 1.
\end{equation}
By the Cauchy-Schwarz equality, $\UG_\psi(A,B) = 1$ iff there is a $\lambda \in \R$ such that
\begin{equation}\label{eq:coherent_state}
(A - \expectation{A}{\psi}) \psi = i \lambda (B - \expectation{B}{\psi}) \psi.
\end{equation}
\end{lemma}
When the operators are generators of unitary representations -- as is the case in this work -- a stronger result has been shown to hold \cite{Kraus1967FurtherRelations,Christensen2004UncertaintyGroups,Folland1997UncertaintySurvey}: $A$ and $B$ will be self-adjoint, and we can replace $[A, B]$ with its densely defined closure. However, this is unnecessary for our purposes.
To maximally disentangle position-orientation, we want a wavelet minimising an $\SE(2)$ UG. We use unitary group representations, which induce skew-symmetric generators that become symmetric upon multiplication with $i$. The next lemma tells us that unitarily equivalent representations yield equivalent UPs.
\begin{lemma}
Let $\mathcal{R}: G \to B(H)$ be a unitary group representation of Lie group $G$ with unit element $e \in G$. The \emph{Lie algebra representation} $\diff \mathcal{R}$ is given by
\begin{equation}\label{eq:lie_algebra_representation}
\diff \mathcal{R}(A)\psi = \lim_{t\to 0}
t^{-1} (\mathcal{R}_{\exp(t A)} - I) \psi, \quad  A \in T_e(G), \psi \in \mathcal{D}(\diff \mathcal{R}(A))
\end{equation}
with $\mathcal{D}(\diff \mathcal{R}(A)) \coloneqq \{\psi \in H \mid \textrm{limit \eqref{eq:lie_algebra_representation} exists}\}$
and Lie group exponential $\exp$.
If $\mathcal{R}_1$ and $\mathcal{R}_2$ are unitarily equivalent, i.e. there is a unitary $U$ such that $\mathcal{R}_1 = U^{-1} \after \mathcal{R}_2 \after U$, then $
\UG_\psi(i \diff \mathcal{R}_1(A), i \diff \mathcal{R}_1(B)) =
\UG_{U \psi}(i \diff \mathcal{R}_2(A), i \diff \mathcal{R}_2(B))
$.
\end{lemma}
\begin{proof}
Since $\lim_{t \to 0} t^{-1} ((\mathcal{R}_1)_{\exp(t A)} - I)\psi = U^{-1} \lim_{t \to 0} t^{-1}((\mathcal{R}_2)_{\exp(t A)} - I) U \psi$, $\expectation{i \diff \mathcal{R}_1(A)}{\psi} = \expectation{i U^{-1} \after \diff \mathcal{R}_2(A) \after U}{\psi} = \expectation{i\diff \mathcal{R}_2(A)}{U \psi}$, so the result follows. \qed
\end{proof}
The representations must additionally be irreducible to ensure existence of a minimiser -- rather than a minimising sequence -- of \eqref{eq:uncertainty_principle}. As such, we \emph{must} use the unitary irreducible representations (UIRs) of $\SE(2)$ on $\Ltwo(S^1)$, 
supporting the polar separability desideratum.
\begin{definition}\label{def:irreducible_representation}
The (only) $\SE(2)$ UIRs $\mathcal{V}^{\vec{\omega}}: \SE(2) \to B(\Ltwo(S^1))$ are given by
\begin{equation}\label{eq:irreducible_representation}
\mathcal{V}_{(\vec{x}, \theta)}^{\vec{\omega}} f(\phi) \coloneqq \exp(-i R_\phi \vec{\omega} \cdot \vec{x}) f(\phi - \theta), \quad f \in \mathbb{L}_{2}(S^1),
\end{equation}
for $\vec{\omega} \in \Rtwo$. If $\vec{\omega}_1 = R\vec{\omega}_2$ for rotation $R \in \SO(2)$, then $\mathcal{V}^{\vec{\omega}_1}$ and $\mathcal{V}^{\vec{\omega}_2}$ are unitarily related by $U f = f \after R^{-1}$, i.e. we just rotate the wavelets.
\end{definition}
We choose $\vec{\omega} \propto (0, 1) \eqqcolon \vec{e}_2$, as we want a directional wavelet, and write $\mathcal{V}^\rho \coloneqq \mathcal{V}^{\rho \vec{e}_2}$.
For the Lie algebra $T_e(\SE(2)) = \se(2)$, we define the basis $A_1 \coloneqq \partial_x|_e$, $A_2 \coloneqq \partial_y|_e$, and $A_3 \coloneqq \partial_\theta|_e$, and compute the corresponding generators:
\begin{equation}\label{eq:generators}
i \diff \mathcal{V}^\rho(A_1) = \rho \cos(\phi),
i \diff \mathcal{V}^\rho(A_2) = \rho \sin(\phi), \textrm{ and }
i \diff \mathcal{V}^\rho(A_3) = -i \partial_\phi.
\end{equation}
The uncertainty gap w.r.t. $i \diff \mathcal{V}^\rho(A_2)$ and $i \diff \mathcal{V}^\rho(A_3)$
is minimised if
\begin{equation}\label{eq:minimum_uncertainty_equation}
\rho \sin(\phi) \Phi(\phi) = i \diff \mathcal{V}^\rho(A_2) \Phi(\phi) \overset{\textrm{Lem.~\ref{lem:uncertainty_principle}}}{=} i \lambda \cdot i \diff \mathcal{V}^\rho(A_3) \Phi(\phi) = i \lambda \cdot -i \partial_\phi \Phi(\phi),
\end{equation}
for some $\lambda \in \R$, or, equivalently $\sin(\phi) \Phi(\phi) = -\lambda \partial_\phi \Phi(\phi)$. It then follows that
$\Phi_\lambda^\mathrm{opt}(\phi) \propto \exp\left(\frac{\cos(\phi)}{\lambda}\right)$,
which is the scaled density of the von Mises distribution. 

Recall from Cor.~\ref{cor:polar_profiles} that we need $\sum_{\theta \in \SO(2, N)} \Phi(\phi - \theta) \approx 1$; ideally the sum would equal $1$, which cannot be done with $\Phi_\lambda^{\mathrm{opt}}$. 
However, B-splines \eqref{eq:cake_wavelets} exactly satisfy this constraint. 
Moreover, as they are defined as densities of sums of i.i.d. random variables, they approximate the density of a wrapped normal distribution: 
as (for $\lambda$ small)
$\exp\left(\frac{\cos(\phi)}{\lambda}\right) \approx \exp\left(\frac{1}{\lambda}\right) \sum_{n \in \Z} \exp\left(-\frac{(\phi - 2n\pi)^2}{2 \lambda}\right) \eqqcolon \Phi_\lambda^\mathcal{N}(\phi)$,
they then also approximate a minimum uncertainty state. Next we provide formal global estimates to corroborate these heuristics.
\begin{proposition}\label{prop:convergence}
The uncertainty gap $\UG_{\Phi_{\lambda}^{\mathcal{N}}}$ of $\Phi_\lambda^{\mathcal{N}}$ is bounded from above by
\begin{equation}\label{eq:ug_bounds_gaussian}
\frac{
  -\frac{\lambda}{\pi} \sinh(\lambda) \partial_{\lambda} (\vartheta_3(e^{-\lambda}))^2
}{\abs*{\Re\left(
    \erf\left(\frac{2 \pi - i \lambda}{\sqrt{4 \lambda}}\right) +
    C_\lambda e^{-\frac{\lambda}{2}} \sqrt{8 \pi \lambda} \left(
        \erf\left(\frac{\pi - i \lambda}{\sqrt{2 \lambda}}\right) -
        \erf\left(\frac{\frac{\pi}{2} - i \lambda}{\sqrt{2 \lambda}}\right)
    \right) -
    2 C_\lambda^2
\right)}^2},
\end{equation}
with elliptic theta function of 3rd kind $\vartheta_3$ and
\begin{equation*}
C_\lambda \coloneqq \sqrt{\frac{\lambda}{2 \pi}}\left(1 + \erf\left(\frac{-1}{\sqrt{2\lambda}}\right)\right) + 2 e^{-\pi^2/2\lambda}.
\end{equation*}
Consequently, $\UG_{\Phi^{\mathcal{N}}_{\lambda}} \to 1$ as $\lambda \downarrow 0$.
\end{proposition}
\begin{proof} 
Upper bound \eqref{eq:ug_bounds_gaussian} follows by estimating $\Phi_\lambda^{\mathcal{N}}$: 
\begin{equation*}
\exp\left(\frac{1}{\lambda}\right) \exp\left(-\frac{\phi^2}{2}\right) \eqqcolon \tilde{\Phi}_\lambda^{\mathcal{N}}(\phi) \leq \Phi_\lambda^{\mathcal{N}}(\phi) \leq \tilde{\Phi}_\lambda^{\mathcal{N}}(\phi) + 2 \int_\pi^\infty \tilde{\Phi}_\lambda^{\mathcal{N}}(\phi) \diff \phi,
\end{equation*}
and substituting these estimates in \eqref{eq:uncertainty_principle}. We see that the upper bound \eqref{eq:ug_bounds_gaussian} converges to $1$ as $\lambda \downarrow 0$, since both the denominator and the numerator tend to $1$ as $\lambda \downarrow 0$: the denominator tends to $1$ since $C_\lambda \to 0$ and error functions tend to $1$; the numerator tends to $1$ since $-\lambda^2 \partial_\lambda (\vartheta_3(e^{-\lambda}))^2 \to \pi$.
Squeeze theorem then yields $\lim_{\lambda \downarrow 0} \UG_{\Phi_\lambda^{\mathcal{N}}} = 1$. \qed
\end{proof}
\begin{theorem}\label{thm:convergence}
Denote the angular profile of cake wavelet \eqref{eq:cake_wavelets} with $N$ orientations and spline order $k$ as $\Phi_{N, k}$. Set $N_\lambda(k) \coloneqq \sqrt{\frac{\pi^2 k}{3 \lambda}}$. Then, 
$\lim \limits_{k \to \infty, \lambda \uparrow 0} \UG_{\Phi_{N_\lambda(k), k}} = 1$.
\end{theorem}
\begin{proof}
The UG is unchanged by scaling the wavelet.
Since $\Uniform[-\frac{1}{2}, \frac{1}{2}]$ has mean $0$ and variance $\sqrt{1/12}$, the CLT gives $\sqrt{\frac{12}{k}} \sum_{i = 1}^k X_i \to \mathcal{N}(0, 1)$ if $X_i \overset{\textrm{i.i.d.}}{\sim} \Uniform[-\frac{1}{2}, \frac{1}{2}]$, so $\lim_{k \to \infty} \sqrt{\frac{(k + 1) \pi}{6}} \Phi_{N_\lambda(k), k} \to \exp(-\frac{1}{\lambda})\Phi_\lambda^{\mathcal{N}}$. 
We exchange limit and UG as the wavelets lie in its domain and all involved operators are closed. The result follows from Prop.~\ref{prop:convergence}. \qed
\end{proof}
Essentially, we must then increase the spline order to converge to the (scaled) density of a wrapped normal distribution $\Phi_\lambda^{\mathcal{N}}$, and send $\lambda \uparrow 0$ to reduce the uncertainty gap between $\Phi_\lambda^{\mathcal{N}}$ and the true minimum uncertainty state $\Phi_\lambda^{\mathrm{opt}}$.

We have verified this convergence in Thm.~\ref{thm:convergence} numerically, cf. Fig.~\ref{fig:convergence}. We indeed see that the UG of the angular profile of cake wavelets \eqref{eq:cake_wavelets} tends to $1$ as we increase the spline order and increase the number of orientations. Hence, for sufficiently high spline orders and numbers orientations, cake wavelets are close to optimal in this minimal uncertainty sense, given the prior constraints imposed the desiderata. For e.g. $k = 3$, $\lambda = 1$, so $N_\lambda(k) < 4$, the UG is already less than $1.1$.
\begin{figure}
\centering
\begin{subfigure}[t]{0.49\linewidth}
\begin{tikzpicture}[baseline=(current bounding box.north west)]
\begin{axis}[
    width=\linewidth,
    xlabel={$\phi$},
    ylabel={$\Phi$},
    cycle list name=color list,
    xtick={-pi, 0, pi},
    xticklabels={$-\pi$, $0$, $\pi$},
    xmin=-pi, xmax=pi,
    ymin=0, ymax=1,
    enlargelimits=false,
    no markers,
]

\definecolor{cvM}{rgb}{0., 0.466667, 0.733333}
\definecolor{cG}{rgb}{0.2, 0.733333, 0.933333}
\definecolor{cC3}{rgb}{0.8, 0.2, 0.0666667}
\definecolor{cC6}{rgb}{0.933333, 0.2, 0.466667}
\definecolor{cC9}{rgb}{0.933333, 0.466667, 0.2}
\definecolor{cC12}{rgb}{0., 0.6, 0.533333}

\addplot[thick, cvM, dashed] table[x index=0, y index=1] {Figures/angular_profiles.dat};
\addplot[thick, cG] table[x index=0, y index=2] {Figures/angular_profiles.dat};
\addplot[thick, cC3] table[x index=0, y index=3] {Figures/angular_profiles.dat};
\addplot[thick, cC6] table[x index=0, y index=4] {Figures/angular_profiles.dat};
\addplot[thick, cC9] table[x index=0, y index=5] {Figures/angular_profiles.dat};
\addplot[thick, cC12] table[x index=0, y index=6] {Figures/angular_profiles.dat};
\end{axis}
\end{tikzpicture}
\caption{Angular profiles $\lambda = 1/5$.}\label{fig:angular_profiles}
\end{subfigure}
\begin{subfigure}[t]{0.49\linewidth}
\begin{tikzpicture}[baseline=(current bounding box.north west)]
\begin{axis}[
    width=\linewidth,
    xlabel={$\lambda$},
    ylabel={$\UG_\Phi$},
    legend pos=north west,
    cycle list name=color list,
    xmin=0, xmax=1,
    ymin=1, ymax=1.1,
    enlargelimits=false,
    no markers,
]

\definecolor{cvM}{rgb}{0., 0.466667, 0.733333}
\definecolor{cG}{rgb}{0.2, 0.733333, 0.933333}
\definecolor{cC3}{rgb}{0.8, 0.2, 0.0666667}
\definecolor{cC6}{rgb}{0.933333, 0.2, 0.466667}
\definecolor{cC9}{rgb}{0.933333, 0.466667, 0.2}
\definecolor{cC12}{rgb}{0., 0.6, 0.533333}

\addplot[ultra thick, cvM, dashed] table[x index=0, y index=1] {Figures/UGs.dat};
\addlegendentry{\tiny $\Phi_\lambda^\mathrm{opt}$}
\addplot[ultra thick, cG] table[x index=0, y index=2] {Figures/UGs.dat};
\addlegendentry{\tiny $\Phi_\lambda^\mathcal{N}$}
\addplot[ultra thick, cC3] table[x index=0, y index=3] {Figures/UGs.dat};
\addlegendentry{\tiny $N_\lambda(3)$}
\addplot[ultra thick, cC6] table[x index=0, y index=4] {Figures/UGs.dat};
\addlegendentry{\tiny $N_\lambda(6)$}
\addplot[ultra thick, cC9] table[x index=0, y index=5] {Figures/UGs.dat};
\addlegendentry{\tiny $N_\lambda(9)$}
\addplot[ultra thick, cC12] table[x index=0, y index=6] {Figures/UGs.dat};
\addlegendentry{\tiny $N_\lambda(12)$}
\end{axis}
\end{tikzpicture}
\caption{Uncertainty gaps.}\label{fig:uncertainty_gap}
\end{subfigure}
\caption{The angular profile $\Phi_{N_\lambda(k), k}$ of cake wavelets \eqref{eq:cake_wavelets} approximates the optimal profile well both in (a) $\Ltwo(S^1)$ and (b) UG cf.~Thm.~\ref{thm:convergence}.}\label{fig:convergence}
\end{figure}
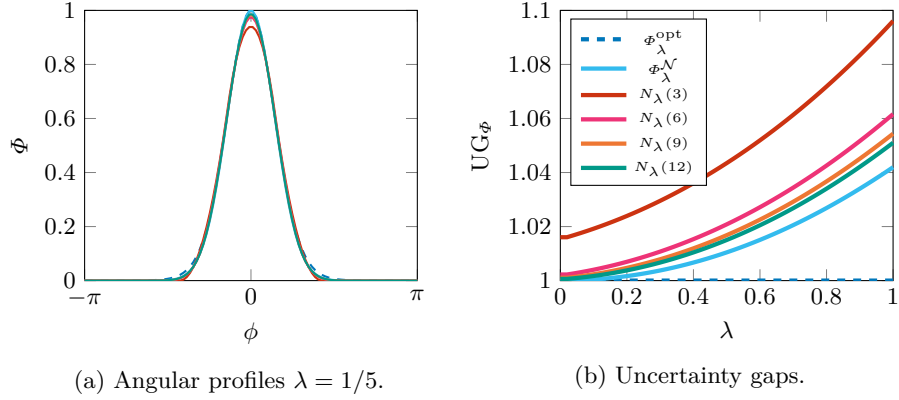


\section{Applications}\label{sec:applications}
Cake wavelets have a proven track record in classical multi-orientation analysis. It is noteworthy, then, that in G-CCNs \cite{Cohen2016GroupNetworks}, and PDE-G-CNNs \cite{Smets2022PDENetworks}, their PDE-based generalisations, data is lifted using trained kernels.
Fixing the lifting layer to use cake wavelets could be particularly valuable for PDE-G-CNNs, where most parameters are found in the lifting layer. The interpretability and connections to neurogeometry will also be improved. 
This was previously proposed in \cite{Bellaard2023GeometricPDEGCNNs}, which claims that fixing the lifting layer need not harm network expressivity. However, \cite{Bellaard2023GeometricPDEGCNNs} assumes that any lifting wavelet can be written as a linear combination of roto-translated cake wavelets; here we prove that this is indeed the case.
\begin{theorem}[No need to train the lifting layer]\label{thm:no_need_to_train_the_lifting_layer}
For any lifting wavelet $\psi^{\mathrm{train}}$, there exist $c^i \in \R$ and $A_i \in \se(2)$ such that -- on $\Ltwo^{\rho_0}(\Rtwo)$ -- we have
\begin{equation}\label{eq:no_need_to_train_the_lifting_layer}
\mathcal{W}_{\psi^{\mathrm{train}}} = \sum_i c^i \mathcal{W}_{\mathcal{U}_{\exp(A_i)}\psi} = \sum_i c^i \mathcal{R}_{\exp(A_i)} \mathcal{W}_{\psi} = \sum_i c^i \exp(\diff \mathcal{R}(A_i)) \mathcal{W}_{\psi},
\end{equation}
with $\mathcal{R}$ the $\SE(2)$ right regular representation given by $\mathcal{R}_g f(h) \coloneqq f(h g)$.
The left-invariant convection generated by $\diff \mathcal{R}(A_i)$ is solved by $\exp(\diff \mathcal{R}(A_i))$, which is precisely what is trained in the convection layer of a PDE-G-CNN \cite{Smets2022PDENetworks}.
\end{theorem}
\begin{proof}
Note that $\mathcal{W}_{\sum_i c^i \psi_i} = \sum_i c^i \mathcal{W}_{\psi_i}$ for any proper wavelets $\psi_i$ and $c^i \in \R$. Trained lifting wavelets lie in the span of $\{\delta_{\vec{x}_i}\}_i$, with $\vec{x}_i$ the locations of the pixels within the kernel. On $\Ltwo^{\rho_0}(\Rtwo)$, $\delta$ coincides with the reproducing kernel (RK) of $\Ltwo^{\rho_0}(\Rtwo)$, $K^{\rho_0}$ \cite{Duits2007ImageGroup}. The fast reconstruction property of cake wavelets gives
\begin{equation*}
\sum_{\theta \in \SO(2, N)} \hat{\psi}(R_\theta^{-1} \cdot) = 1_{B(\vec{0}, \rho)} \implies \sum_{\theta \in \SO(2, N)} \psi(R_\theta^{-1} \cdot) = K_{\vec{0}}^\rho, \textrm{ so that}
\end{equation*}
\begin{equation*}
K_{\vec{x}}^\rho = \mathcal{U}_{(\vec{x}, 0)} \sum_{\theta \in \SO(2, N)} \psi(R_\theta^{-1} \cdot) = \mathcal{U}_{(\vec{x}, 0)} \sum_{\theta \in \SO(2, N)} \mathcal{U}_{(\vec{0}, \theta)} \psi = \sum_{\theta \in \SO(2, N)} \mathcal{U}_{(\vec{x}, \theta)} \psi.
\end{equation*}
As $\exp$ is surjective, we can find $A \in \se(2)$ such that $\exp(A) = (\vec{x}, \theta)$, so
\begin{equation*}
\psi^{\textrm{train}} = \sum_i \tilde{c}^i K_{\vec{x}_i}^\rho = \sum_{\theta \in \SO(2, N)} \sum_i \tilde{c}^i \mathcal{U}_{(\vec{x}_i, \theta)} \psi = \sum_i c^i \mathcal{U}_{\exp(A_i)} \psi,
\end{equation*}
proving the first equality in \eqref{eq:no_need_to_train_the_lifting_layer}. The second equality follows from
\begin{equation*}
\mathcal{W}_{\mathcal{U}_g \psi} f(h) \coloneqq (\mathcal{U}_h \mathcal{U}_g \psi, f) = (\mathcal{U}_{hg} \psi, f) \eqqcolon \mathcal{W}_{\psi} f(hg) \eqqcolon \mathcal{R}_g \mathcal{W}_{\psi} f(h).
\quad \quad  \qed
\end{equation*}
\end{proof}
So a layer that lifts with cake wavelets and performs sufficient left-invariant convections can represent any trained lifting kernel, if the data is in $\Ltwo^{\rho_0}(\Rtwo)$.  

We will now compare fixed and trained lifting layers in (PDE-)G-CNNs experimentally. We implement the networks with the Python package \texttt{lietorch} \cite{Smets2022PDENetworks}.
Each network has six layers: one lifting layer, four convolutional layers, and one projection layer. In each case, the convolutional layers have sixteen channels; the fixed lifting layers use 40 convections. Each architecture is trained and tested ten times on DRIVE, a retinal vessel segmentation dataset \cite{2004StaalDRIVE}. We quantify the performance with the Dice coefficient.
\begin{figure}
\centering
\begin{subfigure}[t]{0.6\textwidth}
\begin{tikzpicture}[baseline=(current bounding box.west)]
\begin{axis}[
    boxplot/draw direction=x,
    height=4.2cm,
    xlabel={Dice Coefficient},
    xtick={0.8, 0.81, 0.82},
    xmin=0.8, xmax=0.82,
    ytick={1, 2, 3, 4},
    yticklabels={Trained Lift G-CNN, Fixed Lift G-CNN, Trained Lift PDE-G-CNN, Fixed Lift PDE-G-CNN}
]

\addplot [
    boxplot prepared={
        lower whisker=0.8047211454871753,
        lower quartile=0.8078151723146143,
        median=0.8090763989328762,
        upper quartile=0.8106601153425078,
        upper whisker=0.8146789523085497,
        draw position=1
    },
    draw=blue,
    thick
] coordinates {};

\addplot [
    boxplot prepared={
        lower whisker=0.8019952798515344,
        lower quartile=0.8026249637651486,
        median=0.8038791005054067,
        upper quartile=0.8049505942324535,
        upper whisker=0.806835370986224,
        draw position=2
    },
    draw=black,
    thick
] coordinates {};

\addplot [
    boxplot prepared={
        lower whisker=0.8110121964060122,
        lower quartile=0.8133721545186188,
        median=0.8139536540297228,
        upper quartile=0.8145108374858883,
        upper whisker=0.816908940738388,
        draw position=3
    },
    draw=black,
    thick
] coordinates {};

\addplot [
    boxplot prepared={
        lower whisker=0.8037191760546583,
        lower quartile=0.8059155932934656,
        median=0.8068801451093228,
        upper quartile=0.8084884533293459,
        upper whisker=0.8099708539299393,
        draw position=4,
    },
    draw=Green,
    thick
] coordinates {};

\end{axis}
\end{tikzpicture}
\end{subfigure}%
\hfill
\begin{subfigure}[t]{0.35\textwidth}
\begin{tabular}{c|c}
Architecture & Parameters \\ \hline
\begin{tabular}{@{}c@{}} Fixed Lift \\ PDE-G-CNN \end{tabular} & {\color{Green}\textbf{2704}} \\ \hline
\begin{tabular}{@{}c@{}} Trained Lift \\ PDE-G-CNN \end{tabular} & 4128 \\ \hline
\begin{tabular}{@{}c@{}} Fixed Lift \\ G-CNN \end{tabular} & 97360 \\ \hline
\begin{tabular}{@{}c@{}} Trained Lift \\ G-CNN \end{tabular} & {\color{blue}98784}
\end{tabular}
\end{subfigure}
\caption{Comparison of (PDE-)G-CNNs tested on DRIVE.}
\label{fig:experiments}
\end{figure}
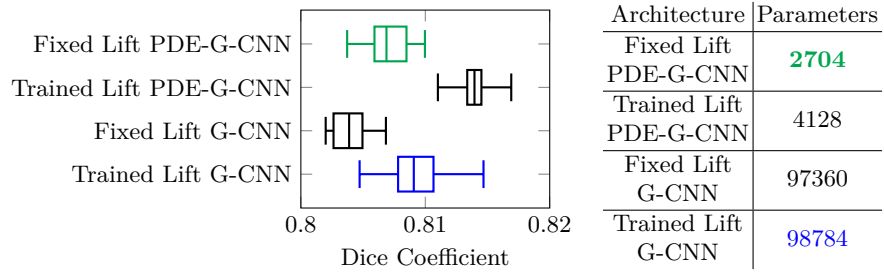
The results are summarised in Fig.~\ref{fig:experiments}. The box plots show that the PDE-G-CNNs perform slightly better than the G-CNNs, and that fixing the lifting layer slightly reduces performance ($<1~\%$). In this case, fixing the lifting layer reduces the number of trainable parameters by 1400; for the PDE-G-CNN this is a sizeable reduction of 34~\%. We have two possible explanations for the slight decrease in performance in spite of Thm.~\ref{thm:no_need_to_train_the_lifting_layer}: (1) the fact that the networks \emph{can} represent the same function does not necessarily mean they \emph{will}, and (2) we assumed the data was disk limited and hence may have lost some important high frequency information.

\subsubsection{Conclusion \& Future Work.}
We showed that for an orientation score with fast reconstruction, position-orientation uncertainty is minimised by wavelets that are well-approximated by cake wavelets (Fig.~\ref{fig:angular_profiles}): the uncertainty gap of the angular profile of cake wavelets converges to the minimum of $1$ when the spline order increases and variance decreases (Thm.~\ref{thm:convergence} and Fig.~\ref{fig:uncertainty_gap}). 
We then showed that we can greatly reduce the number of required parameters and improve the interpretability of (PDE-)G-CNNs by fixing the lifting layer to use cake wavelets, supported by Thm.~\ref{thm:no_need_to_train_the_lifting_layer}, with only slight impact on the model's performance (Fig.~\ref{fig:experiments}). 
In future work we wish to axiomatically derive lifting wavelets for orientation scores of 3D images \cite{Janssen2018DesignImages}.

\subsubsection{Acknowledgements.}
The Dutch foundation of Research (NWO) is gratefully acknowledged for its financial support via VIC.202.031. Furthermore, EAISI is gratefully acknowledged for financial support through the EIDMAR programme.

\bibliographystyle{splncs04}
\bibliography{references}
\end{document}